\documentclass{amsart}
\usepackage{setspace, amsmath, amsthm, amssymb, amsfonts, amscd, epic, graphicx, ulem, dsfont}
\usepackage[T1]{fontenc}
\usepackage{multirow}
\usepackage{bbm}

\newtheorem{thm}{Theorem}[section]
\newtheorem{lem}{Lemma}[section]
\newtheorem{pro}{Proposition}[section]
\newtheorem{cor}{Corollary}[section]

\theoremstyle{abstract}

\theoremstyle{remark}
\newtheorem*{rema}{Remark}

\newtheorem{exa}{\textbf{Example}}

\theoremstyle{definition}

\newcommand{\ran}{\text{\rm{ran}}}

\newcommand{\R}{\mathbb{R}}

\newcommand{\C}{\mathbb{C}}

\makeatletter \@namedef{subjclassname@2010}{
  \textup{2010} Mathematics Subject Classification}
\makeatother

\begin{document}

\title[Operators Similar to Their Adjoint]{Bounded and Unbounded Operators Similar to Their Adjoints}
\author[Dehimi and Mortad]{Souheyb Dehimi and Mohammed Hichem Mortad $^*$}

\address{Department of
Mathematics, University of Oran, B.P. 1524, El Menouar, Oran 31000.
Algeria.\newline {\bf Mailing address}:
\newline Dr Mohammed Hichem Mortad \newline BP 7085
Es-Seddikia\newline Oran
\newline 31013 \newline Algeria}

\email{sohayb20091@gmail.com}

\email{mhmortad@gmail.com, mortad@univ-oran.dz.}

\begin{abstract}
In this paper, we establish results about operators similar to their
adjoints. This is carried out in the setting of bounded and also
unbounded operators on a Hilbert space. Among the results, we prove
that an unbounded closed operator similar to its adjoint, via a
cramped unitary operator, is self-adjoint. The proof of this result
works also as a new proof of the celebrated result by Berberian on
the same problem in the bounded case. Other results on similarity of
hyponormal unbounded operators and their self-adjointness are also
given, generalizing famous results by Sheth and Williams.
\end{abstract}



\thanks{$*$ Corresponding author.}

\maketitle

\section{Introduction}

First, we notice that while we will be recalling most of the
essential background we will assume the reader is familiar with any
other result or notion which will appear in the present paper. Some
of the standard textbooks on bounded and unbounded operator theory
are \cite{CON}, \cite{GGK}, \cite{halmos-book-1982},
\cite{Martin-Putinar-hyponormal operators book}, \cite{RS1},
\cite{RUD}, \cite{SCHMUDG-book-2012} and \cite{WEI}.

The notions of bounded self-adjoint, normal, hyponormal, unitary and
cramped unitary operators are defined in their usual fashion. So are
the notions of unbounded closed, symmetric, self-adjoint, normal and
hyponormal operators. The spectrum and the approximate spectrum of
an operator are denoted respectively by $\sigma$ and $\sigma_a$. We
shall not recall their definitions here.

The numerical range of a bounded operator $T$ on a $\C$-Hilbert
space $H$, denoted by $W(T)$, is defined as
\[W(T)=\{<Tx,x>:~x\in H,~\|x\|=1\}.\]

If $S$ and $T$ are two unbounded operators with domains $D(S)$ and
$D(T)$ respectively, then $S\subset T$ means that $T$ is
 an extension of $S$, that is,
\[D(S)\subset D(T) \text{ and }  \forall x\in D(S):~Sx=Tx.\]
We also assume throughout this paper that all operators are linear
and defined on a separable complex Hilbert space $H$, and that
unbounded operators have a dense domaine (so that the uniqueness of
the adjoint is guaranteed). They are said to be densely defined.

We define the product $ST$ of two unbounded operators with domains
$D(S)$ and $D(T)$ respectively by:
\[(ST)x=S(Tx), ~\forall x\in D(ST)=\{x\in D(T):~Tx\in D(S)\}.\]

Recall that an unbounded operator $S$, defined on a Hilbert space
$H$, is said to be invertible if there exists an \textit{everywhere
defined} (i.e. on the whole of $H$) bounded operator $T$, which then
will be designated by $S^{-1}$, such that
\[TS\subset ST=I\]
where $I$ is the usual identity operator.

Recall also that if $S$, $T$ and $ST$ are all densely defined, then
we have $T^*S^*\subset (ST)^*$. There are cases where equality holds
in the previous inclusion, namely if $S$ is bounded. The next result
gives another case where the equality does hold

\begin{lem}[\cite{WEI}]\label{(AB)*=B*A*} If $S$ and $T$ are densely defined and $T$ is invertible with
inverse $T^{-1}$ in $B(H)$, then $(ST)^* =T^*S^*$.
\end{lem}

The next lemma is also known.

\begin{lem}\label{closed range closed closed oeprator like invertibility inequality}
Let $T$ be a densely defined and closed operator in a Hilbert space
$H$, with domain $D(T)\subset H$. Assume that for some $k>0$,
\[\|Tx\|\geq k\|x\| \text{ for all } x\in D(T).\]
Then $\ran(T)$ is closed.
\end{lem}

Finally, let us recall some other important results for us.

\begin{thm}[\cite{Berb-1962-operators unitarily equivalent to
adjoints}]\label{berberian-cramped-unitary
equivalence-self-adjointness} If $U$ is a cramped unitary element of
$\mathcal{A}$ (where $\mathcal{A}$ is any $B^*$-algebra), and $T$ is
an element of $\mathcal{A}$ such that $UTU^*=T^*$, then $T$ is
self-adjoint.
\end{thm}

\begin{rema}
The previous theorem will be generalized to unbounded operators with
a proof which works in the bounded case too. See Theorem
\ref{Berberian-Dehimi-Mortad-unbounded!} and the remark below it.
\end{rema}

\begin{thm}[\cite{berb-1964-numerical range}]\label{berberian-T invertible cramped}
Let $T$ be a bounded operator for which $0\not\in \overline{W(T)}$.
Then $T$ is invertible and the unitary operator
$T(T^*T)^{-\frac{1}{2}}$ is cramped.
\end{thm}

\begin{thm}[\cite{Sheth-PAMS-hyponormal} or \cite{Wil}]\label{sheth-Williams-bounded-similarity}
Let $T$ be a bounded hyponormal operator. If $S$ is any bounded
operator for which $0\not\in \overline{W(S)}$, then
\[ST=T^*S\Longrightarrow T=T^*.\]
\end{thm}

\begin{pro}[\cite{Janas-Hyponormal-unbd-III}]\label{Janas-hyponormal-III-PROPOSITION}
Let $T$ be an unbounded, closed and hyponormal operator in some
Hilbert space $H$. Then $W(T)\subset\text{conv}\sigma(T)$, where
conv$\sigma(T)$ denotes the the convex hull of $\sigma(T)$.
\end{pro}

\begin{pro}[\cite{OtaSchm}]\label{Ota-schmud-prosoition-quasi-similar-symmetric-self-adj}
Let $T$ be a densely defined, closed and symmetric operator in a
Hilbert space. If $T$ is quasi-similar to its adjoint $T^*$, then
$T$ is self-adjoint (for the definition of quasi-invertibility, the
reader may look at \cite{OtaSchm}).
\end{pro}

The notion of similarity of operators is important from matrices on
finite-dimensional vector spaces to unbounded operators on a Hilbert
space. Many authors have worked on this type of problems for bounded
operators. We refer the interested reader to \cite{BeckPut},
\cite{Berb-1962-operators unitarily equivalent to adjoints},
\cite{berb-1964-numerical range}, \cite{DePrima-correcting-Singh},
\cite{EMB}, \cite{Jeon et al-2008-IEOT},
\cite{Sheth-PAMS-hyponormal}, \cite{Singh-Mangla-operators-inverses
similar to their adjoint} and \cite{Wil}.

There have been some works on similar unbounded operators but only a
few compared to those in the bounded case. This is due probably to
the complexity of the domains involved. Some of those papers are
\cite{Mortad-Tsukuba-2010}, \cite{MHM.BBMSSS}, \cite{OtaSchm} and
\cite{STO}.

In the present article, we establish some new results on similarity
in the setting of bounded and unbounded operators on a Hilbert
space. We have two main sections, one dedicated to bounded operators
and the other is devoted to unbounded operators.

\section{Main Results: The Bounded Case}

The main result of this section is the following. It permits us to
drop the hypothesis of hyponormality in Theorem
\ref{sheth-Williams-bounded-similarity} at the cost of imposing a
commutativity-like assumption.

\begin{thm}\label{Main Theorem: Bounded Case}
Let $S,T$ be two bounded operators satisfying: $S^{-1}T^*S=T$,
$S^*ST=TS^*S$ and $0\not\in \overline{W(S)}$. Then $T$ is
self-adjoint.
\end{thm}

\begin{proof}
Since $0\not\in \overline{W(S)}$, $S$ is invertible. So, let $S=UP$
be its polar decomposition. Remember that $P=(S^*S)^{\frac{1}{2}}$
is positive and $U=S(S^*S)^{-\frac{1}{2}}$ is unitary. By Theorem
\ref{berberian-T invertible cramped}, $U$ is even cramped.

Since $S^*ST=TS^*S$, we have
\[ P^2T=TP^2\text{ or } PT=TP.\]
Hence we may write
\begin{align*}
&S^{-1}T^*S=T\\
\Longleftrightarrow& P^{-1}U^*T^*UP=T \\
\Longleftrightarrow& U^*T^*U=PTP^{-1}\\
\Longleftrightarrow& U^*T^*U=TPP^{-1}\\
\Longleftrightarrow& U^*T^*U=T\\
\Longleftrightarrow& T^*=UTU^*.\\
\end{align*}
As $U$ is cramped, Theorem \ref{berberian-cramped-unitary
equivalence-self-adjointness} applies and yields the
self-adjointness of $T$, establishing the result.
\end{proof}

Before giving another result on similarity, it appears to be
convenient to recall the following result here:

\begin{thm}[Singh-Mangla, \cite{Singh-Mangla-operators-inverses similar to their adjoint}]\label{similarity-unitarity-Singh-Mangla-1973-PAMS-WRONG}
If $T$ is an invertible normaloid operator such that
$T^*=S^{-1}T^{-1}S$, where $0\not\in \overline{W(S)}$, then $T$ is
unitary.
\end{thm}

Shortly afterwards, DePrima \cite{DePrima-correcting-Singh} found
out that Theorem
\ref{similarity-unitarity-Singh-Mangla-1973-PAMS-WRONG} was actually
false by giving a counterexample! Then DePrima
\cite{DePrima-correcting-Singh} gave some extra conditions for the
conclusion of Theorem
\ref{similarity-unitarity-Singh-Mangla-1973-PAMS-WRONG} to hold. One
of the results there is the following:

\begin{thm}[DePrima, \cite{DePrima-correcting-Singh}]\label{similarity-unitarity-Deprima}
Let $T$ be an invertible normaloid or convexoid operator such that
$T^{-1}$ too is normaloid or convexoid. If $T^*=S^{-1}T^{-1}S$,
where $0\not\in \overline{W(S)}$, then $T$ is unitary.
\end{thm}

Our result in this spirit is the following:

\begin{thm}
Let $S$ and $T$ be two bounded operators such that $TS^*S=S^*ST$ and
$S^{-1}T^*S=T^{-1}$, where $0\not\in \overline{W(S)}$ and where also
$T$ is invertible, then $T$ is unitary.
\end{thm}

\begin{proof}
Let $S=UP$ where $U$ is unitary and $P$ is positive (where
$P=(S^*S)^{\frac{1}{2}}$). We then have
\[TS^*S=S^*ST\Longrightarrow S^*ST^{-1}=T^{-1}S^*S,\]
hence
\[P^2T^{-1}=T^{-1}P^2 \text{ so that } PT^{-1}=T^{-1}P.\]
Therefore,
\begin{align*}
&S^{-1}T^*S=T^{-1}\\
\Longleftrightarrow  &P^{-1}U^*T^*UP=T^{-1}\\
\Longleftrightarrow &U^*T^*U=PT^{-1}P^{-1}\\
\Longleftrightarrow &U^*T^*U=T^{-1}PP^{-1}\\
\Longleftrightarrow &U^*T^*U=T^{-1}\\
\Longleftrightarrow &T^*=UT^{-1}U^*.
\end{align*}
Since $0\not\in \overline{W(S)}$, $U$ is cramped so that Theorem 2
of \cite{Singh-Mangla-operators-inverses similar to their adjoint}
applies and gives us $T^*=T^{-1}$, completing the proof.
\end{proof}

\section{Main Results: The Unbounded Case}

The first result of the this section is the following (it
generalizes Theorem \ref{sheth-Williams-bounded-similarity} to an
unbounded operator setting).

\begin{thm}\label{Unbounded hyponormal Main theorem!!}
Let $S$ be a bounded operator on a $\C$-Hilbert space $H$ such that
$0\not\in \overline{W(S)}$. Let $T$ be an unbounded and closed
hyponormal operator with domain $D(T)\subset H$. If $ST^*\subset
TS$, then $T$ is self-adjoint.
\end{thm}

\begin{rema}
We did add an extra condition (namely closedness) on $T$ as regards
to Theorem \ref{sheth-Williams-bounded-similarity}. This is fine for
closed operators are considered as the natural substitutes of the
bounded ones. Besides, if $T$ is not closed, then it cannot be
self-adjoint.
\end{rema}

Now, we prove Theorem \ref{Unbounded hyponormal Main theorem!!}.

\begin{proof}
The proof is divided into three claims:
\begin{enumerate}
  \item \textbf{Claim 1:} $\sigma_a(T^*)=\sigma(T^*)$. By
  definition, $\sigma_a(T^*)\subset \sigma(T^*)$. To show the
  reverse inclusion, let $\lambda\not\in \sigma_a(T^*)$. Then there
  exists a positive number $k$ such that
  \[\|T^*x-\lambda x\|\geq k\|x\| \text{ for all } x\in D(T^*).\]
  Hence $T^*-\lambda I$ is clearly injective. Besides, and by Lemma \ref{closed range closed closed oeprator like invertibility inequality}, $\ran(T^*-\lambda I)$ is
  closed as $T^*-\lambda I$ is closed for $T^*$ is so. Now, since $T$ is hyponormal, so is $T-\overline{\lambda} I$
  which means that
  \[\|Tx-\overline{\lambda}x\|\geq \|T^*x-\lambda x\|\geq k\|x\| \text{ for all } x\in D(T)\subset D(T^*).\]
  Whence $T-\overline{\lambda}I$ is also one-to-one so that
  \[\ran(T^*-\lambda I)^{\perp}=\ker(T-\overline{\lambda}I)=\{0\} \text{ or } \overline{\ran(T^*-\lambda I)}=H.\]
  Thus $T^*-\lambda I$ is onto since we already observed that its range was
  closed. Therefore, $\lambda\not\in \sigma(T^*)$.
  \item \textbf{Claim 2:} $\sigma(T)\subset\R$. Let
  $\lambda\in\sigma(T^*)=\sigma_a(T^*)$. Then for some $x_n\in
  D(T^*)$ such that $\|x_n\|=1$ we have $\|T^*x_n-\lambda x_n\|\rightarrow 0$ as
  $n\rightarrow\infty$. Since $ST^*\subset
TS$ and $x_n\in D(T^*)$, we have $ST^*x_n=TSx_n$ so that we may
write the following
  \begin{align*}
  0\leq |(\overline{\lambda}-\lambda)<Sx_n,x_n>|=&|<(ST^*S^{-1}-\lambda+\overline{\lambda}-T)Sx_n,x_n>|\\
  \leq&|<(S(T^*-\lambda I)x_n,x_n>|+|<(\overline{\lambda}-T)Sx_n,x_n>|\\
  \leq & \|S\|~\|T^*x_n-\lambda x_n\|+|<Sx_n,(\lambda-T^*)x_n>|\\
  \leq &\|S\|~\|T^*x_n-\lambda x_n\|+\|S\|~\|T^*x_n-\lambda x_n\|\\
  =& 2\|S\|~\|T^*x_n-\lambda x_n\|\rightarrow 0.
  \end{align*}
  (where in the second inequality we used the fact that $x_n\in D(T^*)$ and $Sx_n\in
  D(T)$ both coming from $ST^*\subset TS$).
  However, the condition $0\not\in \overline{W(S)}$ forces us to
  have $\lambda=\overline{\lambda}$. Accordingly,
  $\sigma(T^*)\subset\R$ or just $\sigma(T)\subset\R$ (remember that
  $\sigma(T^*)=\{\overline{\lambda}:~\lambda\in\sigma(T)\}$).
  \item \textbf{Claim 3:} $T=T^*$. Since
  $\sigma(T)\subset\R$, Proposition
  \ref{Janas-hyponormal-III-PROPOSITION} implies that
  $W(T)\subset\R$, which clearly implies that $<Tx,x>\in\R$ for all $x\in D(T)$, which, in its turn, means that $T$ is symmetric.
  Hence $T$ is quasi-similar to $T^*$ via $S$ and $I$,  so that
  Proposition
  \ref{Ota-schmud-prosoition-quasi-similar-symmetric-self-adj}
  applies and gives the self-adjointness of $T$.
  This completes the proof.
  \end{enumerate}
\end{proof}

The condition $ST^*\subset TS$ in the foregoing theorem is not
purely conventional, i.e. we may not obtain the desired result by
 merely assuming instead that $ST\subset T^*S$, even with a slightly
stronger condition (i.e. symmetricity in lieu of hyponormality).
This is seen in the following proposition

\begin{pro}
There exist a bounded operator $S$ such that $0\not\in
\overline{W(S)}$, and an unbounded and closed hyponormal $T$ such
that $ST\subset T^*S$ whereas $T\neq T^*$.
\end{pro}

\begin{proof}
Consider any unbounded, densely defined, closed and symmetric
operator $T$ which is \textit{not self-adjoint}. Let $S=I$, i.e. the
identity operator on the Hilbert space. Then $S$ is bounded and
$0\not\in \overline{W(S)}$. Also, $T$ is closed and hyponormal.
Finally, it is plain that
\[T=ST\subset T^*=T^*S.\]
\end{proof}

\begin{rema}
We have not insisted on the explicitness of the example $T$ in the
previous proof. This was not too important. Besides, there are many
of them. For instance, the interested reader may just consult
Exercise 4 on page 316 of \cite{CON}.
\end{rema}

We can still obtain the self-adjointness of $T$ from $ST\subset
T^*S$ by imposing an extra condition on $T$. We have

\begin{thm}\label{Unbounded hyponormal Invertible theorem!!}
Let $S$ be a bounded operator such that $0\not\in \overline{W(S)}$.
Let $T$ be an unbounded hyponormal and invertible operator. If
$ST\subset T^*S$, then $T$ is self-adjoint.
\end{thm}

\begin{proof}
Since $T$ is invertible with an everywhere defined bounded inverse,
we have
\[ST\subset T^*S\Longrightarrow ST^{-1}=(T^{-1})^*S.\]
Since $T$ is hyponormal, the bounded $T^{-1}$ too is hyponormal (see
\cite{Janas-Hyponormal-unbd-I}). Hence by
\cite{Sheth-PAMS-hyponormal}, $T^{-1}$ is self-adjoint. Hence
\[T^{-1}=(T^{-1})^*\Longrightarrow T(T^{-1})^*=I\Longrightarrow T^*\subset T.\]
Now, since $T$ is hyponormal, $D(T)\subset D(T^*)$ so that finally
we have $T=T^*$, that is, $T$ is self-adjoint.
\end{proof}

The condition of invertibility in the foregoing theorem may not
simply be dispensed with. This is illustrated by the following
example:

\begin{exa}
Let $A$ be an unbounded operator defined on a Hilbert space $H$,
with domain $D(A)\subsetneqq H$. Set $T=A-A$, then $T$ is not closed
and hence it is surely not self-adjoint. Finally, let $S=I$ the
identity operator on $H$. Now we verify that the remaining
conditions (except for invertibility) of the theorem are fulfilled.
\begin{enumerate}
  \item $T$ is hyponormal for $T^*=0$ with $D(T^*)=H\supset D(A)=D(T)$ so that
  \[\|T^*x\|=\|Tx\|=0 \text{ for all } x\in D(T).\]
  \item Since $S=I$, obviously $0\not\in\overline{W(S)}$. Moreover,
  \[T=0_{D(A)}\subset T^*=0_{H} \text{ so that } ST\subset T^*S.\]
\end{enumerate}
\end{exa}

Last but not least, we have a very nice and important result which
generalizes Theorem \ref{berberian-cramped-unitary
equivalence-self-adjointness} to unbounded operators.

\begin{thm}\label{Berberian-Dehimi-Mortad-unbounded!}
Let $U$ be a cramped unitary operator. Let $T$ be a closed operator
such that $UT=T^*U$. Then $T$ is self-adjoint.
\end{thm}

\begin{proof}
First we prove that $U^2T=TU^2$. Since $U$ is bounded and
invertible, we have
\[(UT)^*=T^*U^* \text{ and } (T^*U)^*=U^*T^{**}=U^*\overline{T}=U^*T\]
(by Lemma \ref{(AB)*=B*A*}). Hence $T^*U^*=U^*T$. We may then write

\begin{align*}
U^2T{U^{*}}^2=&U(UTU^*)U^*\\
=&UT^*U^*\\
=&UU^*T\\
=&T,
\end{align*}
giving $U^2T=TU^2$ or $T{U^{*}}^2={U^{*}}^2T$ or $U^2T^*=T^*U^2$.

Next, we prove that $TU=UT^*$. We have
\begin{align*}
TU=&U^*T^*UU\\
=&U^*T^*U^2\\
=&U^*UUT^*\\
=&UT^*.
\end{align*}
Hence also $U^*T^*=TU^*$.

The penultimate step in the proof is to prove that $T$ is normal. To
this end, set $S=\frac{1}{2}(U+U^*)$. Following \cite{Wil}, $S>0$.

Then we show that $STT^*\subset T^*TS$. We have
\[UTT^*=T^*UT^*=T^*TU\]
and
\[U^*TT^*=T^*U^*T^*=T^*TU^*.\]
Hence
\begin{align*}
STT^*&=\frac{1}{2}(U+U^*)TT^*\\
&=\frac{1}{2}UTT^*+\frac{1}{2}U^*TT^* \text{ (as $U$ is bounded)}\\
&=\frac{1}{2}T^*TU+\frac{1}{2}T^*TU^*\\
&\subset T^*TS.
\end{align*}
So according to Corollary 5.1 in \cite{STO}, $TT^*=T^*T$, and
remembering that $T$ is taken to be closed, we immediately conclude
that $T$ is normal. Accordingly, and by Corollary 3 in
\cite{Mortad-Tsukuba-2010},
\[UT=T^*U\Longrightarrow T=T^*\]
as $0\not\in \overline{W(U)}$, establishing the result.
\end{proof}

\begin{rema}
Evidently, a hypothesis like $UT\subset T^*U$ would not yield the
desired result. For example, take $T$ to be any symmetric and closed
unbounded operator $T$ which is not self-adjoint. Let $U=I$ be the
identity operator on the given Hilbert space. Then clearly
$UT\subset T^*U$ while $T\neq T^*$.

The assumption $U$ being cramped is indispensable even in the
bounded case. This was already observed by Beck-Putnam
\cite{BeckPut} and McCarthy \cite{McCarthy-Beck-Putnam}.
\end{rema}

\begin{rema}
Going back to the previous proof, we observe that this proof may
well be applied to bounded operators. Hence we have just given a new
proof of Theorem \ref{berberian-cramped-unitary
equivalence-self-adjointness} which bypasses the Cayley transform.
\end{rema}

Thanks to Theorem \ref{Berberian-Dehimi-Mortad-unbounded!}, we may
prove an unbounded version of Theorem \ref{Main Theorem: Bounded
Case}. It reads

\begin{cor}\label{berberian-dehimi-mortad-corollary}
Let $S$ be a bounded operator and $T$ be an unbounded closed
operator satisfying: $S^{-1}T^*S=T$, $S^*ST=TS^*S$ and $0\not\in
\overline{W(S)}$. Then $T$ is self-adjoint.
\end{cor}

\begin{proof}The same proof as that of Theorem \ref{Main Theorem: Bounded
Case}, mutatis mutandis.
\end{proof}

\bibliographystyle{amsplain}

\end{document}